\documentclass[a4paper]{amsart}
\usepackage[utf8]{inputenc}

\usepackage{amsmath}
\usepackage{amssymb}
\usepackage{lmodern}
\usepackage{tikz-cd}
\usepackage{enumitem}
\usepackage{bbm}
\usepackage{todonotes}
\usepackage{lmodern}
\usepackage[T1]{fontenc}

\usepackage[hidelinks]{hyperref}

\newcommand{\axiom}[1]{\mathsf{#1}}

\newcommand{\ZFA}{\axiom{ZFA}}

\newcommand{\ZF}{\axiom{ZF}}

\newcommand{\HS}{\axiom{HS}}
\newcommand{\DC}{\axiom{DC}}

\DeclareMathOperator{\sym}{sym}
\DeclareMathOperator{\fix}{fix}

\DeclareMathOperator{\tcl}{tcl}
\DeclareMathOperator{\id}{id}

\theoremstyle{plain}
\newtheorem{thm}{Theorem}
\newtheorem{lemma}[thm]{Lemma}
\newtheorem{prop}[thm]{Proposition}
\newtheorem*{cor}{Corollary}
\newtheorem{claim}{Claim}

\newtheorem*{mainthm}{Main Theorem}
\theoremstyle{definition}
\newtheorem{definition}[thm]{Definition}

\newtheorem{qstn}[thm]{Question}

\author{Asaf Karagila}
\author{Jonathan Schilhan}
\email{karagila@math.huji.ac.il}
\urladdr{http://karagila.org}
\email{j.schilhan@leeds.ac.uk}
\urladdr{http://www.logic.univie.ac.at/~schilhanj/}
\address{School of Mathematics,
    University of Leeds.
    Leeds, LS2~9JT, UK}
\thanks{The authors were supported by a UKRI Future Leaders Fellowship [MR/T021705/2]. No data are associated with this article.}

\date{\today}

\title{Geometric condition for Dependent Choice}

\begin{document}

\begin{abstract}
We provide a geometric condition which characterises when the Principle of Dependent Choice holds in a Fraenkel--Mostowski--Specker permutation model. This condition is a slight weakening of requiring the filter of groups to be closed under countable intersections. We show that this condition holds nontrivially in a new permutation model we call ``the nowhere dense model'' and we study its extensions to uncountable cardinals as well.
\end{abstract}

\maketitle

\section{Introduction}
Permutation models are models of set theory with atoms.\footnote{Traditionally either by weakening Extensionality (to allow ``multiple empty sets'') or Regularity (to allow Quine atoms, i.e.\ $x=\{x\}$).} These models are used often to prove basic independence results related to the Axiom of Choice. The technique was originally developed by Fraenkel, improved upon by Mostowski by introducing the concept of supports, and the modern presentation is due to Specker using filters of subgroups.\footnote{For a more detailed historical review see the remarks at the end of Chapter~4 in \cite{Jech:AC}.}

With these models it is relatively easy to construct models where the Axiom of Choice fails, and using a myriad of transfer theorems we can translate some of these results to the context of Zermelo--Fraenkel. Regardless of their intended use, permutation models are still an interesting way to exhibit the tension between symmetries and definablity on the one hand and the Axiom of Choice on the other.

We rarely want to violate the Axiom of Choice so much that it is become completely useless. Indeed, often times the goal is to preserve a fragment of the Axiom of Choice. One of the important fragments is the \emph{Principle of Dependent Choice} ($\DC$), which has many equivalents and reformulations throughout mathematics (e.g.\ the Baire Category Theorem), and can be succinctly stated as ``Every tree without maximal nodes has an infinite chain''.

Reading the standard proof separating Dependent Choice from other principles (see Chapter~8 of \cite{Jech:AC}) it is easy to conjecture a permutation model satisfies $\DC$ if and only if the filter of subgroups is countably complete. In this work we show that this is not quite the right notion, and we can in fact weaken the completeness to what we call ``shift completeness'' of the filter.
\section{Preliminaries}
Let us recall the basics of permutation models, following \cite{Jech:AC}. The underlying theory we are working with is $\ZFA$, i.e.\ $\ZF$ together with atoms. For simplicity we will also always assume that the collection of atoms is a set $A$. Models of this theory have a natural von Neumann like hierarchy where instead of starting with the empty set, we start with the set of atoms $A$ and then successively build powersets and unions at limit steps. Whenever $\pi$ is a permutation of $A$, $\pi$ recursiveley extends via this hierarchy to the whole universe by stipulating that $$\pi(x) = \{ \pi(y) : y \in x \}.$$

Now let $\mathcal{G}$ be a group of permutations of $A$. A \emph{filter of subgroups} of $\mathcal{G}$ is a non-empty set $\mathcal{F}$ of subgroups of $\mathcal{G}$ that is closed under supergroups and finite intersections. It is called \emph{normal} if it is also closed under conjugation by elements of $\mathcal{G}$, i.e.\ for every $H \in \mathcal{F}$ and $\pi \in \mathcal{G}$, $\pi H \pi^{-1} \in \mathcal{F}$. We say that $\mathcal{F}$ is $\kappa$-complete, for a cardinal $\kappa$, if whenever $\gamma<\kappa$ and $\{H_\alpha\mid\alpha<\gamma\}\subseteq\mathcal{F}$, then $\bigcap_{\alpha<\gamma}H_\alpha\in\mathcal{F}$.

In the context of some fixed $\mathcal{G}$ and a normal filter $\mathcal{F}$ of subroups of $\mathcal{G}$, we say that a set $x$ is \emph{symmetric} if there is a group $H \in \mathcal{F}$ such that $\pi(x) = x$ for every $\pi \in H$. In other words, $\sym(x) := \{ \pi \in \mathcal{G} : \pi(x) = x \} \in \mathcal{F}$. Moreover we say that $x$ is \emph{hereditarily symmetric} if every element of the transitive closure $\tcl(\{x\} \cup x)$ is symmetric. The class of hereditarily symmetric sets is then denoted $\HS$.

\begin{lemma}[Theorem~4.1 in {\cite{Jech:AC}}]
$\HS \subseteq V$ is a transitive model of $\ZFA$.
\end{lemma}

We will say that $\HS$ is the \emph{permutation model} obtained from $\mathcal{G}$ and $\mathcal{F}$.

\begin{definition}
\label{def:stratclosed}
Let $\mathcal{F}$ be a normal filter of subgroups of $\mathcal{G}$. We say that $\mathcal{F}$ is \emph{shift complete} if for any decreasing sequence $\langle H_n : n \in \omega \rangle$ in $\mathcal{F}$, there are $\pi_n \in \mathcal{G}$, for $n \in \omega$, such that for every $n \in \omega$, $\pi_n \in K_n$ and $\bigcap_{n \in \omega} K_n \in \mathcal{F}$, where $K_n$ is defined as $(\pi_{n -1 } \circ \dots \circ \pi_0) H_n (\pi_0^{-1} \circ \dots \circ \pi_{n-1}^{-1})$.
\end{definition}
To improve the readability, we will always use $K_n$ to denote the shifted sequence and we will use $\sigma_n$ to denote the composition $\pi_{n-1}\circ\dots\circ\pi_0$. Assuming $\DC$, it is true that $\mathcal{F}$ is shift complete if and only if it has a basis which is shift complete (requiring that $\bigcap_{n\in\omega}K_n$ contains a basis element, in this case). This is similar to the case of normality, although we need to rely on a modicum of choice since we need to choose infinitely many basis elements at once.

\begin{definition}
Let $A$ be a collection of atoms, let $\mathcal{G}$ be group of permutations of $A$ and $\mathcal{F}$ be a normal filter of subgroups of $\mathcal{G}$. Then the \emph{essential subfilter} of $\mathcal{F}$ is the filter generated by $\{ \sym(x) : x \in \HS \}$.
\end{definition}

It is easy to see that the essential subfilter is itself normal and produces the same permutation model, and in typical constructions $\mathcal{F}$ is itself already essential.
\begin{prop}
Assuming $\DC$ holds in $V$, if $\mathcal{F}$ is shift complete, then the essential subfilter is shift complete.
\end{prop}
\begin{proof}
  Let $\langle H_n : n\in\omega\rangle$ be a sequence in the essential subfilter, and choose $x_n\in\HS$ such that $\sym(x_n)=H_n$ for every $n\in\omega$. Using the shift completeness of $\mathcal{F}$, there is $\langle\pi_n: n\in\omega\rangle$ such that $\bigcap_{n\in\omega}K_n\in\mathcal{F}$.

  Define $y_n=\sigma_n(x_n)$, then we have that $\sym(y_n)=\sigma_n\sym(x_n)\sigma_n^{-1}=K_n$. Therefore $\sym(\langle y_n:n\in\omega\rangle) = \bigcap_{n \in \omega} K_n$, $\langle y_n: n\in\omega\rangle\in\HS$, and $\bigcap_{n \in \omega} K_n$ is in the essential subfilter, as wanted.
\end{proof}

\section{Dependent Choice is Essentially shift completeness}
\begin{mainthm}\label{thm:main}
Assume $\DC$ holds in $V$ and let $\mathcal{F}$ be a normal filter of subgroups of $\mathcal{G}$. Then the following are equivalent:
\begin{enumerate}
    \item The essential subfilter of $\mathcal{F}$ is shift complete.
    \item $\HS \models \DC$.
\end{enumerate}
\end{mainthm}

\begin{proof}
(1) implies (2): Let $(T, \leq) \in \HS$ be a tree without a maximal element. Since $\DC$ holds in $V$, there is an increasing sequence $\langle s_n : n \in \omega \rangle$ in $T$. Let $H_0 := \sym(T,\leq)$ and $H_{n+1} := \sym(T,\leq, \langle s_i : i \leq n \rangle)$ for every $n \in \omega$. Then each $H_n$ is in the essential subfilter of $\mathcal{F}$ and they form a decreasing sequence. According to (1), there are $\pi_n \in \mathcal{G}$ such that $K := \bigcap_{n \in \omega} K_n\in\mathcal{F}$. Further let $t_n := \sigma_{n}(s_n)$, for every $n \in \omega$. Then  $t_n \in \HS$ and \begin{align*}
    \sym(t_n) &= \sym(\sigma_{n}(s_n)) \\ &= \sigma_{n} \sym(s_n) \sigma_{n}^{-1} \\ &\supseteq \sigma_{n} H_{n+1} \sigma_{n}^{-1} \\ &= K_{n+1} \supseteq K.
\end{align*}  Moreover, as $(s_n, s_{n+1}) \in {\leq}$ and $\sigma_{n+1} \in H_0 = \sym(\leq)$, we have that \begin{align*}
    \sigma_{n+1}(s_n, s_{n+1}) &= (\sigma_{n+1}(s_n), \sigma_{n+1}(s_{n+1})) \\ &= (\pi_{n+1} \circ \sigma_{n}(s_{n}), \sigma_{n+1}(s_{n+1})) \\ &= (\pi_{n+1}(t_n), t_{n+1}) \\ &= (t_n, t_{n+1}) \in {\leq}.
\end{align*}

The last equality follows since $\pi_{n+1} \in K_{n+1} \subseteq \sym(t_n)$. Thus $\langle t_n : n \in \omega \rangle$ is increasing in $T$ and $\langle t_n : n \in \omega \rangle \in \HS$.

(2) implies (1): Let $\langle H_n : n \in \omega \rangle$ be in the essential subfilter. Using $\DC$ in $V$, for every $n \in \omega$, let $x_n \in \HS$ be such that, without loss of generality, $\sym(x_n)=H_{n+1}$.

Define $T := \{ \pi(\langle x_i : i < n \rangle) : \pi \in H_0, n \in \omega \}$. Since $H_0 \subseteq \sym(T)$ we have that $T\in\HS$, and as a tree ordered by inclusion $T$ has no maximal elements. Thus, by (2), there is a branch $\langle t_n : n \in \omega \rangle \in \HS$ in $T$. Again applying $\DC$ in $V$, we find $\tau_{n}$ such that $\langle t_i : i \leq n \rangle = \tau_{n}(\langle x_i : i \leq n \rangle)$ for every $n$. Let $\pi_0 = \tau_0$ and for each $n \geq 1$, let $\pi_n = \tau_n \circ \tau_{n-1}^{-1}$ and observe that $\pi_n \circ \dots \circ \pi_0 = \tau_n$. We claim that using $\langle \pi_n : n \in \omega\rangle$ is a shifting sequence for $\langle H_n : n\in\omega\rangle$. We let $K_0 = H_0$ and $K_n := \tau_{n-1} H_n \tau_{n-1}^{-1}$, in anticipation that $\bigcap_{n\in\omega}K_n\in\mathcal{F}$.

\begin{claim}
For every $n \in \omega$, $\pi_n \in K_n$.
\end{claim}

\begin{proof}
For $n=0$ this is trivial, as $K_0 = H_0$ and $\pi_0 = \tau_0 \in H_0$. For $n \geq 1$, since $K_n = \tau_{n-1} H_n \tau_ {n-1}^{-1} \supseteq \sym(\tau_{n-1} (\langle x_i : i <n \rangle))$, it suffices to show that $\pi_n$ fixes $\tau_{n-1} (\langle x_i : i <n \rangle)$. This is true since \[\pi_n(\tau_{n-1}(\langle x_i : i <n \rangle))= \tau_n(\langle x_i : i <n \rangle) \subseteq \tau_n(\langle x_i : i \leq n \rangle) = \langle t_i : i \leq n \rangle.\] As $\tau_n(\langle x_i : i <n \rangle)$ has length $n$, \[\tau_n(\langle x_i : i <n \rangle) = \langle t_i : i < n \rangle = \tau_{n-1}(\langle x_i : i <n \rangle).\qedhere\]
\end{proof}

Let $K := \sym(\langle t_n : n \in \omega \rangle) \cap K_0$.

\begin{claim}
$K \subseteq \bigcap_{n \in \omega} K_n$.
\end{claim}

\begin{proof}
Let $\pi \in K$ and $n \in \omega$ be arbitrary. Then $\pi \in K_0$ and for $n\geq 1$, since $K_n = \tau_{n-1} H_n \tau_{n-1}^{-1} \supseteq \sym(\tau_{n-1} (\langle x_i : i <n \rangle))$, it suffices to show that $\pi$ fixes $\tau_{n-1} (x_{n-1})$. This is similar to the previous claim.
\end{proof}

Since $K$ is in the essential subfilter of $\mathcal{F}$, this completes the proof.
\end{proof}

\section{Nowhere dense model}
  We consider the rational numbers, $\mathbb{Q}$, with their linear order as our structure. The group $\mathcal{G}$ is the group of order automorphisms (i.e.\ order preserving bijections). For any subset $E\subseteq\mathbb{Q}$, we let $\fix(E) := \{\pi \in \mathcal{G} : \pi\restriction E=\id\}$. Let $\mathcal{F}$ be the filter generated by $\{ \fix(E) : E\subseteq\mathbb{Q} \text{ is nowhere dense} \}$. It is not hard to see that any singleton is nowhere dense, so the filter is certainly not countably complete.

  We claim that $\mathcal{F}$ is shift complete. Let $E_n$ be an increasing sequence of nowhere dense subsets of $(\mathbb{Q},<)$. Let $\langle I_n : n \in \omega \rangle$ enumerate all open intervals of $\mathbb{Q}$. We will recursively define automorphisms $\pi_n$ and non-empty intervals $J_n \subseteq I_n$ such that
  \begin{enumerate}
  \item $\pi_n \in \fix(\sigma_n`` E_n) =\sigma_n \fix(E_n)\sigma_n^{-1}$ and
  \item $\bigcup_{n \in \omega} \sigma_n`` E_n \cap \bigcup_{n \in \omega} J_n = \varnothing$.
  \end{enumerate}


  \begin{lemma}
    Let $E$ be a nowhere dense set and $J=\bigcup_{i<n}(a_i-\varepsilon_i,b_i+\varepsilon_i)$ be a disjoint union of open intervals for some $\varepsilon_i>0$. There is an automorphism, $\pi\in\fix(\mathbb{Q}\setminus J)$, and $\pi``E\cap(a_i-\frac{\varepsilon_i}4,b_i+\frac{\varepsilon_i}4)=\varnothing$ for $i<n$.
  \end{lemma}
  \begin{proof}
    For each $i<n$ we define an automorphism $\tau_i\in\fix(\mathbb{Q}\setminus(a_i-\varepsilon_i,b+\varepsilon_i))$ and take $\pi$ to be the composition of these automorphisms, since the intervals are disjoint this composition is commutative. For readability let us omit $i$ from the subscript, as we are working on each of the intervals separately.

    Since $E$ is nowhere dense, we can find $(c,d)\subseteq (a,b)$ such that $E\cap (c,d)=\varnothing$. Let $\tau$ be an automorphism of $(a-\varepsilon,b+\varepsilon)$ such that $\tau(a)=a-\frac\varepsilon2$, $\tau(c)=a-\frac\varepsilon4$, $\tau(d)=b+\frac\varepsilon4$, and $\tau(b)=b+\frac\varepsilon2$. Then $\tau``E\cap(a-\frac\varepsilon4,b+\frac\varepsilon4)=\varnothing$, as wanted.
  \end{proof}

  Pick $J_0 = (a_0,b_0)\subseteq I_0$ such that for some $\varepsilon_0 > 0$, $(a_0-\varepsilon_0,b_0+\varepsilon_0)\cap E_0=\varnothing$. Since $E_1$ is nowhere dense we can apply the lemma to obtain $\pi_0\in\fix(E_0)$ such that $\pi_0``E_1\cap(a_0-\frac{\varepsilon_0}4,b_0+\frac{\varepsilon_0}4)=\varnothing$.

  Suppose that we have defined $J_i=(a_i,b_i)$ and $\pi_i$ for $i<n$ such that for some $\varepsilon_i>0$ we have $(a_i-\varepsilon_i,b_i+\varepsilon_i)\cap\sigma_n``E_n=\varnothing$. It is important to note that the sequence of $\varepsilon_i$ may be taken to be different at each step. We want to find $J_n=(a_n,b_n)\subseteq I_n$ such that for some $\varepsilon_n$ we have $(a_n-\varepsilon_n,b_n+\varepsilon_n)$ is disjoint from $\sigma_n``E_n$ and the previously chosen intervals, by perhaps shrinking the $\varepsilon_i$ even more in order to apply the lemma to $J=\bigcup_{i\leq n}(a_i-\varepsilon_i,b_i+\varepsilon_i)$ and $\sigma_n``E_{n+1}$. Certainly, since $\sigma_n``E_n$ is nowhere dense that requirement is easy to fulfil; if we cannot fulfil the second requirement, then we can choose $J_n$ to be contained in one of the $J_i$ for $i<n$ and we can apply the lemma. This ensures that $\bigcup_{n\in\omega}\sigma_n``E_n\cap\bigcup_{n\in\omega}J_n=\varnothing$ as wanted.

\section{Generalised versions of Dependent Choice and shift completeness}

We can generalise $\DC$ to higher cardinals in the following way. We say that a tree $T$ is $\kappa$-closed if every chain of order type $<\kappa$ has an upper bound. Then $\DC_\kappa$ states that every $\kappa$-closed tree has a chain of order type $\kappa$ or a maximal element. It is not hard to verify, in this formulation, that if $\lambda<\kappa$, then $\DC_\kappa$ implies $\DC_\lambda$ holds as well. We write $\DC_{<\kappa}$ to denote $\DC_\lambda$ holds for all $\lambda<\kappa$. In the case where $\kappa=\lambda^+$ this is just $\DC_\lambda$, and if $\kappa$ is singular, then $\DC_{<\kappa}$ implies $\DC_\kappa$. However, for inaccessible cardinals $\DC_{<\kappa}$ is indeed weaker than $\DC_\kappa$.

\begin{definition}
 Let $\delta$ be an infinite ordinal. Then we say that $\mathcal{F}$ is \emph{$\delta$-shift complete} if for any sequence $\langle H_\alpha : \alpha < \gamma \rangle$ with $\gamma<\delta$ in $\mathcal{F}$, there are $\sigma_\alpha \in \mathcal{G}$, for $\alpha < \gamma$, such that
\begin{enumerate}
    \item  $\bigcap_{\alpha < \gamma} K_\alpha \in \mathcal{F}$, where $K_\alpha =\sigma_\alpha H_\alpha \sigma_\alpha^{-1}$,
    \item for every $\alpha < \beta < \gamma$, $\sigma_\beta \sigma_\alpha^{-1} \in K_\alpha$.
\end{enumerate}
\end{definition}
It is not hard to see that shift-complete as we previously defined is $\omega+1$-shift complete.

\begin{lemma}
   Let $\mathcal{F}$ be a filter of subgroups of $\mathcal{G}$. If $\mathcal{F}$ is $\gamma+1$-shift complete, then $\mathcal{F}$ is $|\gamma|$-complete. If $\mathcal{F}$ is $\gamma+2$-shift complete, then $\mathcal{F}$ is $|\gamma|^+$-complete.
\end{lemma}
\begin{proof}
  For simplicity, we assume that $\gamma=|\gamma|$. To see that $\mathcal{F}$ is $\gamma$-complete, let $\langle H_\alpha : \alpha < \beta \rangle$ be in $\mathcal{F}$, where $\beta < \gamma$. Next extend this sequence arbitrarily to $\langle H_\alpha : \alpha < \gamma \rangle$, for example by repeating $\mathcal{G}$ after $\beta$. Let $\sigma_\alpha$ and $K_\alpha$ be as in the definition of shift complete, for each $\alpha\leq\gamma$. Then we have that $\bigcap_{\alpha < \beta} K_\alpha \in \mathcal{F}$. On the other hand, for each $\alpha < \beta$, $\sigma_\beta \sigma_\alpha^{-1} \in K_\alpha$, so
\begin{align*}
    \bigcap_{\alpha < \beta} H_\alpha &=  \sigma_\beta^{-1} \left(\bigcap_{\alpha < \beta} \sigma_\beta H_\alpha \sigma_\beta^{-1} \right) \sigma_\beta \\ &= \sigma_\beta^{-1} \left(\bigcap_{\alpha < \beta} \sigma_\beta \sigma_\alpha^{-1} K_\alpha  \sigma_\alpha \sigma_\beta^{-1} \right) \sigma_\beta \\ &= \sigma_\beta^{-1} \left(\bigcap_{\alpha < \beta} K_\alpha \right) \sigma_\beta\in \mathcal{F}.
\end{align*}

The proof in the case of $\gamma+2$-shift completeness is similar, but we can now use a sequence of length $\gamma$ to begin with, thus proving that $\mathcal{F}$ is $\gamma^+$-complete.
\end{proof}
We saw with the nowhere dense model that the above theorem is the best we can get, since it is possible to get a filter of groups that is $\omega+1$-shift complete, but not $\omega_1$-complete.

It is a standard observation that if $\mathcal{F}$ is $\kappa$-complete and $\DC_{<\kappa}$ holds in $V$, then $\DC_{<\kappa}$ holds in $\HS$. The proof of this observation actually shows more, it shows that $\HS$ is closed under $\gamma$-sequences for any $\gamma<\kappa$. Indeed, if $X=\{x_\alpha\mid\alpha<\gamma\}\subseteq\HS$, then $\bigcap_{\alpha<\gamma}\sym(x_\alpha) = H\in\mathcal{F}$, and it is not hard to check that $H$ fixes $X$ pointwise, so $X\in\HS$.

\begin{thm}
  Assuming $\DC_\kappa$ holds in $V$, if $\mathcal{F}$ is $\kappa+1$-shift complete, then $\DC_\kappa$ holds in $\HS$.
\end{thm}
\begin{proof}
Suppose that $(T,\leq) \in \HS$ is a $\kappa$-closed tree with no maximal elements. We have that $T$ is $\kappa$-closed in $V$, since there is some $\gamma<\kappa$ and a chain in $T$ of order type $\gamma$ which is not in $\HS$. But by the lemma above, $\mathcal{F}$ is $\kappa$-complete, so this is impossible. Thus there is a branch $\langle s_\alpha : \alpha < \kappa \rangle \in V$ in $T$. Now we proceed exactly as in the proof of Theorem~\ref{thm:main} to get a shifted branch in $\HS$.
\end{proof}

\begin{cor}
  Assuming $\DC_{<\kappa}$ holds in $V$, if $\mathcal{F}$ is $\kappa$-shift complete, then $\DC_{<\kappa}$ holds in $\HS$.\qed
\end{cor}

\section{Open questions}
  From the work of Blass in \cite{Blass:1987} and \cite{Blass:2011} we know that there is a complete characterisation of when the Boolean Prime Ideal theorem holds in a permutation model. The property used by Blass is called a ``Ramsey filter'', but upon deeper inspection it seems to have a strong finitary nature which is at odds with the infinitary nature of shift completeness.
\begin{qstn}
  Is there a natural example of a permutation model where the filter which is both Ramsey and shift complete?
\end{qstn}

It seems somewhat unlikely that $\kappa+1$-shift completeness will be equivalent to $\DC_\kappa$ holding in $\HS$ for uncountable $\kappa$. It seems reasonable to expect that $\HS$ might not be closed under $\gamma$-sequences for all $\gamma<\kappa$, but $\DC_\kappa$ still holds there.
\begin{qstn}
  What is the ``correct'' generalisation of shift completeness which does not imply $\kappa$-completeness?
\end{qstn}
\vspace{1em}
\subsection*{Acknowledgements} The authors would like to thank the anonymous referee for their comments and suggestions.
\newpage
\providecommand{\bysame}{\leavevmode\hbox to3em{\hrulefill}\thinspace}
\providecommand{\MR}{\relax\ifhmode\unskip\space\fi MR }
\providecommand{\MRhref}[2]{%
  \href{http://www.ams.org/mathscinet-getitem?mr=#1}{#2}
}
\providecommand{\href}[2]{#2}

\end{document}